\documentclass[11pt]{article}

\usepackage{arxiv}

\usepackage{amsmath, amssymb, enumerate, amsthm}

\usepackage{longtable}
\usepackage{siunitx}
\usepackage{color}
\usepackage{bm}
\usepackage[colorlinks,allcolors=blue]{hyperref}
\usepackage{natbib}

\usepackage{booktabs}
\usepackage{longtable}
\usepackage{array}
\usepackage{multirow}
\usepackage{wrapfig}
\usepackage{float}
\usepackage{colortbl}
\usepackage{pdflscape}
\usepackage{tabu}
\usepackage{threeparttable}
\usepackage{threeparttablex}
\usepackage[normalem]{ulem}
\usepackage{makecell}

\title{Principal Loading Analysis}

\author{
 Jan O. Bauer
   \And
 Bernhard Drabant  
  \And 
  \\[0.1pt]
   Baden-Wuerttemberg Cooperative State University Mannheim\\ 
 Coblitzallee 1-9\\
 68163 Mannheim, Germany\\
  \texttt{jan.bauer@dhbw-mannheim.de; bernhard.drabant@dhbw-mannheim.de}
}

\newcommand{\bigO}{\mathcal{O}_p}

\newcommand{\R}{\mathbb{R}}

\newcommand{\E}{{\rm E}}
\newcommand{\Var}{{\rm Var}}
\newcommand{\Cov}{{\rm Cov}}
\newcommand{\corr}{{\rm Corr}}

\newcommand{\diag}{{\rm diag}}

\newcommand{\cind}{\overset{\text{c}}{\ind}}
\newcommand{\epsind}{\overset{\varepsilon}{\ind}}
\newcommand{\ind}{\perp \!\!\! \perp}

\newcommand{\PC}{\stackrel{P}{\longrightarrow}}

\newcommand{\Xb}{\bm{X}}
\newcommand{\xb}{\bm{x}}
\newcommand{\vb}{\bm{v}}
\newcommand{\vbhat}{\hat{\bm{v}}}
\newcommand{\vbtilde}{\tilde{\bm{v}}}
\newcommand{\vbtildehat}{\hat{\tilde{\bm{v}}}}
\newcommand{\yb}{\bm{y}}

\newcommand{\Ab}{\bm{A}}

\newcommand{\Vb}{\bm{V}}
\newcommand{\Vbhat}{\hat{\bm{V}}}
\newcommand{\Vbtilde}{\tilde{\bm{V}}}
\newcommand{\Vbtildehat}{\hat{\tilde{\bm{V}}}}

\newcommand{\xibtilde}{\tilde{\bm{\xi}}}
\newcommand{\xitilde}{\tilde{\xi}}

\newcommand{\betab}{\bm{\beta}}

\newcommand{\Sigmab}{\bm{\Sigma}}

\newcommand{\Sigmabhat}{\hat{\bm{\Sigma}}}

\newcommand{\Sigmabtilde}{\tilde{\bm{\Sigma}}}
\newcommand{\Sigmabtildehat}{\hat{\tilde{\bm{\Sigma}}}}
\newcommand{\sigmatilde}{\tilde{\sigma}}
\newcommand{\sigmatildehat}{\hat{\tilde{\sigma}}}

\newcommand{\lambdatilde}{\tilde{\lambda}}
\newcommand{\lambdatildehat}{\hat{\tilde{\lambda}}}
\newcommand{\Lambdabtilde}{\tilde{\bm{\Lambda}}}
\newcommand{\Lambdabtildehat}{\hat{\tilde{\bm{\Lambda}}}}

\newcommand{\Deltab}{\bm{\Delta}}

\newcommand{\lambdahat}{\hat{\lambda}}

\newcommand{\Lambdabhat}{\hat{\bm{\Lambda}}}

\newcommand{\jN}{{j|N}}

\newcommand{\eone}{\textbf{e}_1}
 
\newcommand{\emm}{\textbf{e}_M}

\newcommand{\jth}{j^{\text{th}}}

\newcommand{\Etab}{\bm{\mathrm{H}}}
\newcommand{\etab}{\bm{\eta}}

\newcommand{\Epsilonb}{\bm{\mathrm{E}}}
\newcommand{\epsilonb}{\bm{\varepsilon}}

\renewcommand{\keywords}[1]{\emph{Keywords:} #1}

\newcommand{\mathsubjclassSingle}[2]{\emph{#1:} #2}

\newtheorem{Definition}{Definition}[section]
\newtheorem{Remark}{Remark}[section]
\newtheorem{Theorem}{Theorem}[section]
\newtheorem{Lemma}{Lemma}[section]

\newtheorem{Algorithm}{Algorithm}[section]

\newtheorem*{assumption*}{\assumptionnumber}
\providecommand{\assumptionnumber}{}
\makeatletter
\newenvironment{assumption}[2]
 {%
  \renewcommand{\assumptionnumber}{Assumption #1$^{#2}$}%
  \begin{assumption*}%
  \protected@edef\@currentlabel{#1}%
 }
 {%
  \end{assumption*}
 }
\makeatother

\newcommand{\refNeudecker}{Theorem 6 in \citep{KN93}}

\newcommand{\refWessel}{Theorem 1 in \citep{NW90}}

\numberwithin{equation}{section} 
\begin{document}
\maketitle
\begin{abstract}
This paper proposes a tool for dimension reduction where the dimension of the original space is reduced: the principal loading analysis. principal loading analysis is a tool to reduce dimensions by discarding variables. The intuition is that variables are dropped which distort the covariance matrix only by a little. Our method is introduced and an algorithm for conducting principal loading analysis is provided. Further, we give bounds for the noise arising in the sample case.
\\[12pt]
\keywords{Component Loading, Dimensionality Reduction, Matrix Perturbation Theory, Principal Component Analysis}
\\[12pt]
\mathsubjclassSingle{2010 Mathematics Subject Classification}{62H25, 62J10, 15A18, 15A42, 65F15}
\end{abstract}

\section{Introduction}
\label{s:Introduction}
\noindent
When data is of high dimension it is often beneficial to reduce dimension. Two ways of reducing dimension exist: either by transforming the variables to a reduced set of variables or by selecting a subset of the existing variables \citep{CC08,MPH09}. We propose a new approach pursuing the latter one despite adopting ideas from principal component analysis (PCA). \\
Classical PCA was formulated by \citep{PE01,HO33} and has been extended over the years while still being an active field of research. The basic goal is to transform a set of variables into a subspace spanned by orthogonal variables containing most of its original variance (see for example \citep{JO02} for an overview). One extension to a non-linear approach using kernels was given by \citep{SSM98}, to the so called kernel PCA. Further, \citep{ZHT06} proposed a sparse principal component analysis (sparse PCA) to reduce the dimension as well as the number of used variables. This is done by implementing a further restriction in the underlying maximization problem. Regarding the purpose of discarding variables, \citep{JO72,JO73} presented four methods and provided examples for artificial as well as real data respectively. The methods are based on the idea to select variables that are highly present in the largest eigenvectors, or to discard variables that are highly present in the smallest eigenvectors.\footnote{\emph{Note:} We will denote the eigenvectors associated with the largest eigenvalues as \textit{"largest eigenvectors"}. The intuition behind \textit{"smallest eigenvectors"} is analogue.} In a regression context, \citep{BM94} propose an iterative method to select the covariates using PCA and \citep{MW77} to discard variables considering the respective increase in the residual sum of squares.\\
\citep{KN93} provided a broad overview of existing theories regarding the asymptotics of eigenvectors and eigenvalues of sample correlation and covariance matrices respectively. \citep{DPR82} pointed out that the assumption of simple eigenvalues is needed because the orthonormal basis of the eigenmanifold of eigenvalues with a multiplicity greater than one can be obtained by rotation hence asymptotics for the corresponding eigenvalues are problematic to obtain.\\
However, to our knowledge, there has been no effort in developing a technique where dimension is reduced by selecting a subset of the observed variables based on non-impact in the eigenvectors. Our main contribution is a method as such: the principal loading analysis (PLA). We investigate the underlying form of the sample covariance matrix needed to conduct PLA where we take the presence of perturbations caused by a small sample size and due to the fuzziness of PLA into account. We further provide Algorithms to conduct PLA in practice\\  
The rest of the paper is organised as follows: Section~\ref{s:setup} provides assumptions needed for the remainder of this work. In Section~\ref{s:Metho}, we recap the methodology of PCA and explain the idea of PLA. Our focus lies on Section~\ref{s:BlockPLA}, where we introduce the PLA method and the underlying covariance matrix structure needed for PLA. We will provide bounds for the sample counterparts essential for PLA: for the sample covariance, sample eigenvectors and sample eigenvalues. We suggest using a cut-off threshold for application and Section~\ref{s:Decision} complements this step of PLA. We give recommendations for threshold values in Section~\ref{s:simulationstudy} based on simulation studies. We give simulated examples in Section~\ref{s:DataAnalysis} and take a resume as well as suggest extensions in Section~\ref{s:Conclusion}.
\section{Setup}
\label{s:setup}
\noindent
We first state some notation, assumptions and lemmas used throughout this work. We consider $\xb = \begin{pmatrix} \xb_1 & \cdots & \xb_M \end{pmatrix} \in \R^{N\times M}$ to be an independent and identically distributed (IID) sample containing $n \in\{ 1,\ldots,N\}$ observations of a random vector $\Xb = (X_1 \;\cdots\; X_M)$ with covariance matrix $\Sigmab = (\sigma_{i,j})$ for some $i,j\in\{1,\ldots,M\}$. We also consider the case when the covariance matrix is slightly perturbed with
$$\Sigmabtilde \equiv \Sigmab + \Epsilonb  $$
where $\Epsilonb = (\varepsilon_{i,j})$ is a sparse matrix. $\Epsilonb$ is a technical construction and contains small components we want to extract from $\Sigmab$. Hence, $\varepsilon_{i,j}\neq0\Rightarrow\sigma_{i,j}=0$. The sample counterpart $\Sigmabtildehat$ is of the form
$$\Sigmabtildehat \equiv (\sigmatildehat_{i,j}) \equiv \Sigmab + \Epsilonb  + \Etab_N $$
where $\Etab_N = (\eta_{i,j|N})$ is a perturbation in the form of a random noise matrix. The noise is due to having only a finite number of observations in the sample. We consider the eigendecomposition of $\Sigmabtilde$ to be given by 
\begin{equation}
\label{eq:Eigendec}
\Sigmabtilde \equiv \Vbtilde\Lambdabtilde\Vbtilde^\top
\end{equation}
with $\Vbtilde^\top\Vbtilde=\bm{I}$, $\Lambdabtilde = \diag(\lambdatilde_1,\ldots,\lambdatilde_M)$ and $\lambdatilde_1\geq\ldots\geq\lambdatilde_M$. The eigenvectors $\Vbtilde = \begin{pmatrix} \vbtilde_1  & \cdots & \vbtilde_M\end{pmatrix}$ are ordered according to the respective eigenvalues. The eigendecomposition for $\Sigmabtildehat$ is denoted analogously. The corresponding sample eigenvalues and eigenvectors are given by
$$  \lambdatildehat_j \equiv \lambda_j + \varepsilon_j^\lambda + \eta_j^\lambda  $$
and
$$\vbtildehat_j \equiv \vb_j + \epsilonb_{j} + \etab_{j|N} = \begin{pmatrix} v^{(1)}_{j} \\ \vdots \\ v^{(M)}_{j} \end{pmatrix} + \begin{pmatrix} \varepsilon_{j|N}^{(1)} \\ \vdots \\ \varepsilon_{j|N}^{(M)} \end{pmatrix} + \begin{pmatrix} \eta_{j|N}^{(1)} \\ \vdots \\ \eta_{j|N}^{(M)} \end{pmatrix}  $$
respectively for some $j\in\{1,\ldots,M\}$. For a vector $\vb = (v^{(1)} \;\cdots\; v^{(M)})^\top\in\R^M$ and for $0<p<\infty$, we use the $\mathcal{L}_0$, $\mathcal{L}_p$ and $\mathcal{L}_\infty$ vector norms as $\| \vb \|_0 \equiv |\{m:v^{(m)} \neq 0 \}|$, $\| \vb \|_p \equiv \left( \sum_{m} |v^{(m)}|^p\right)^{1/p}$ and $\|\vb\|_\infty \equiv \max_m|v^{(m)}|$ respectively. For a matrix $\Ab$, $\|\Ab\|_F$ denotes the Frobenius norm. "$\PC$" denotes convergence in probability where we always consider the limit for $N\to\infty$. $\bigO$ denotes stochastic boundedness and $\vb = \bigO(\cdot) \Leftrightarrow \|\vb\|_\infty = \bigO(\cdot)$.\footnote{We will not restate the definition of stochastic boundedness explicitly in this work but refer to common textbooks such as \citep{BFH07} Definition 14.4-3 and 14.4-4.} When two blocks of random variables $X_{i_1},\ldots,X_{i_I}$ and $X_{j_1},\ldots,X_{j_J}$ are uncorrelated, we write $(X_{i_1},\ldots,X_{i_I})\cind (X_{j_1},\ldots,X_{j_J})$. Due to the sparse perturbation $\Epsilonb$, we introduce the following definition:
\begin{Definition}
We say that two blocks of random variables $X_{i_1},\ldots,X_{i_I}$ and $X_{j_1},\ldots,X_{j_J}$ are $\varepsilon$-uncorrelated if $ \sigmatilde_{i_{\bar{i}},j_{\bar{j}}} = \varepsilon_{i_{\bar{i}},j_{\bar{j}}}$ for $\bar{i}\in\{1,\ldots,I\}$ and $\bar{j}\in\{1,\ldots,J\}$. We then write $(X_{i_1},\ldots,X_{i_I})\epsind (X_{j_1},\ldots,X_{j_J})$.
\end{Definition}\noindent
Given the notation above, the following assumption are made for $i,j\in\{1,\ldots,M\}$:
\begin{assumption}{1}{}
\label{ass:1}
$\Sigmabtilde$ has distinct eigenvalues $\lambdatilde_1 > \ldots > \lambdatilde_M$.
\end{assumption}\noindent
Assumption~\ref{ass:1} is needed to apply \refNeudecker. In presence of eigenvalues with an algebraic order larger than one it is problematic to obtain asymptotic results since the corresponding orthonormal basis of the eigenmanifold can be obtained by rotation. Hence we rule out this case for $\Sigmab$ in Assumption~\ref{ass:1}. We refer to \citep{KN93} for an detailed overview of asymptotic results and to \citep{DPR82} for a more elaborate explanation why an algebraic order of one is needed for the eigenvalues of concern.
\begin{assumption}{2}{}
\label{ass:2}
$\E\lbrace (\bm{X} - \E\bm{X})^4 \rbrace < \infty$.
\end{assumption}\noindent
We assume $\Xb$ to be non-explosive.
The following two lemmata provide the methodology for the noise. $\Sigmabtildehat$ is distorted by noise but provides consistent estimators in the limit. The intuition behind this is trivial: we can not assume that our sample represents $\Xb$ perfectly in the first place. The deviation between the sample values and the population values is described by the noise terms. However, the deviation vanishes when the sample size increases.
\begin{Lemma}\label{l:1}
For $i,j \in\{1,\ldots,M\}$ it holds that $\eta_{i,j|N} = \bigO\left(1/\sqrt{N}\right)$
\end{Lemma}
\begin{proof}
$\eta_{i,j|N} = \bigO\left(1/\sqrt{N}\right)$ is an immediate result of \refWessel.
\end{proof}\noindent
We then can further provide bounds for the sample eigenvalues:
\begin{Lemma} \label{l:2}
For $j \in\{1,\ldots,M\}$ it holds that $\eta_{j|N}^\lambda = \bigO\left(1/N^w\right)$, $w \in[\tfrac{1}{2},\infty)$.
\end{Lemma}
\begin{proof}
From Weyl's inequality (Corollary 4.10 in \citep{SS90}) and Lemma~\ref{l:1} we can conclude that $\max_{j} | \hat{\tilde{\lambda}}_j - \tilde{\lambda}_j |  \leq \| \Etab_N \|_F = \bigO\left(1/\sqrt{N}\right)$ hence $\eta_{j|N}^\lambda = \bigO\left(1/N^w\right)$ with $w \in[\tfrac{1}{2},\infty)$.
\end{proof}
\section{Background and Methodology}
\label{s:Metho}
\noindent
Firstly, we recap PCA in this section because some steps of PCA correspond with PLA. Afterwards, we introduce the intuition of PLA which is complemented by its applied parts in Section~\ref{s:BlockPLA}.\\
PCA is a tool for dimension reduction while containing most of the variance within the data. Every observation $n$ is projected into a $K$-dimensional subspace, with $K \leq M$, spanned by the eigenvectors of the covariance matrix denoted loadings. The projected observations then are represented by the so called PCs. Clearly, it is most appealing when $K \ll M$ and $K \leq 3$ for graphic purposes. The underlying decomposition is the eigendecomposition $\Sigmabtilde = \Vbtilde\Lambdabtilde \Vbtilde^\top$ in (\ref{eq:Eigendec}). The identity of PCA is then based on
$$\Xb = \Vbtilde \xibtilde = \Vbtilde_{K} \xibtilde_{K} + \Vbtilde_{M-K}\xibtilde_{M-K}  \; , $$
see \citep{MO12} among others. Here, $\xibtilde = \begin{pmatrix} \xitilde^{(1)} & \cdots &  \xitilde^{(M)} \end{pmatrix}^\top$ is the vector of PCs. $\xibtilde_K =\begin{pmatrix} \xitilde^{(1)} & \cdots &  \xitilde^{(K)} \end{pmatrix}^\top$ denotes the PCs corresponding to the $K$ largest eigenvectors and $\xibtilde_{M-K} =\begin{pmatrix} \xitilde^{(K+1)} & \cdots &  \xitilde^{(M)} \end{pmatrix}^\top$ denotes the PCs corresponding to the $M-K$ smallest eigenvectors respectively. The notation for $\Vbtilde_{K} = \begin{pmatrix} \vbtilde_1  & \cdots & \vbtilde_K\end{pmatrix}$ and $\Vbtilde_{M-K}$ goes in an analogue manner. \\
Considering the geometric shape of a covariance matrix, the components of the eigenvectors reflect the distortion in each dimension. When the, say, $\jth$ component of an eigenvector contains a small value, we can conclude that the $\jth$ variable affects the observations projected into a subspace containing this very eigenvector only a little (if at all) for this particular eigenvector-axis. Further, the importance of each eigenvector-axis is given by the size of the corresponding eigenvalue. This is due to the eigendecomposition
$$ \Var\left(\xitilde^{(j)}\right) = \Var\left(\vbtilde_{j}^\top\Xb \right) = \vbtilde_{j}^\top \Sigmabtilde \vbtilde_{j} = {\lambdatilde}_j \; . $$
Hence, the explained variance by the first $K$ PCs is percentaged given by
\begin{equation}
\label{eq:expvar} \dfrac{ \sum_{m=1}^K \lambdatilde_m }{ \sum_{m=1}^M \lambdatilde_m }\;\; .
\end{equation}
\noindent
When the $\jth$ component of each of the eigenvectors (except for one) is small, the $\jth$ variable affects the projection into the subspace containing most of the variance marginally. Our motivation for PLA is then to discard this very variable instead of projecting all observations into the subspace spanned by the $K$ largest eigenvectors also containing the $\jth$ variable. Therefore, PLA is a tool to detect variables that do not account for any or only for little distortion of the variance and to discard them. \\
PLA proceeds as follows: Assuming that the data lies in the Euclidean space spanned by the unit vectors $\eone,\ldots,\emm$ and that the rows $j_1,\ldots,j_{M^\ast}$ of $M-M^\ast$ eigenvectors contain only values below a threshold (in absolute terms) hence do not distort the covariance matrix, we consider to discard the corresponding random variables $X_{j_1},\ldots,X_{j_{M^\ast}}$ hence transform the data into a subspace spanned by $\textbf{e}_{j}$ with $j\neq j_1,\ldots,j_{M^\ast}$. Therefore, we do not merge the original variables by transforming them into a subspace spanned by the different eigenvectors as done in PCA but rather contain the original variables except the discarded one(s). However, the size of the eigenvalues corresponding to the $M-M^\ast$ eigenvectors has to be considered according to (\ref{eq:expvar}) since the $X_{j_1},\ldots,X_{j_{M^\ast}}$ might explain a fair proportion of the variance. \\
Note that we link variables to eigenvectors and their corresponding eigenvalues hence link a set of variables to a set of eigenvectors with corresponding eigenvalues. Keeping this in mind, a further interpretation of $\Epsilonb$ is outlined in the following remark.
\begin{Remark}
The eigenvalues can be interpreted as estimators of the explained variance biased by $\Epsilonb$ since
$$ \max\limits_j  | \lambdatildehat_j - \lambda_j |  \stackrel{\text{Weyl}}{\leq} \| \Epsilonb + \Etab_N \|_F \PC  \| \Epsilonb \|_F \; $$
where convergence in probability results from Lemma~\ref{l:1} and the Continuous Mapping Theorem. The bias is, however, small due to the sparseness of $\Epsilonb$ and vanishes when the variables are uncorrelated.
\end{Remark}
\section{Principal Loading Analysis}
\label{s:BlockPLA}
\noindent
In this section we provide an explicit algorithm to execute PLA and we investigate the underlying structure of the eigendecomposition of the sample covariance matrix needed to conduct PLA when the sample covariance matrix is of the form $\Sigmabtildehat = \Sigmab + \Epsilonb + \Etab_N$ meaning in the presence of sparse perturbation and noise. We propose that variables can be discarded when small components occur in the population eigenvectors because small entries reflect that the corresponding variable does not distort the covariance matrix in this very direction. In fact, we do look for $\varepsilon$-uncorrelated blocks of variables which only deform towards a few dimensions. However, those variables corresponding to the blocks are only discarded if the sum of the respective eigenvalues is small hence explains only little of the overall variance within the covariance matrix. Considering $\Epsilonb = \bm{0}$, non-zero components of concern can be due to noise within the sample or if the components are in fact different from zero. To get an intuition, we provide bounds for the relevant eigenvectors which have to hold under PLA, as well as bounds for the eigenvalues reflecting the explained variance.\\
To get a first intuition, we start by consider $\Epsilonb = \bm{0}$ again. It holds that variables cause zero components in the eigenvectors when the variables are uncorrelated in blocks, in a way that
\begin{equation}
\label{eq:blocks}
\underbrace{ (X_1,\ldots,X_{M_1}) }_{\kappa_1\text{-many}} \cind  (X_{M_1+1},\ldots,X_{M_2})   \cind \ldots \cind \underbrace{  (X_{M_{L-1}+1},\ldots,X_{M_L})  }_{\kappa_L\text{-many}}
\end{equation}
for an arbitrary $L>1$. Considering $\Epsilonb \neq \bm{0}$, $\Epsilonb$ captures small correlations among the random variables which results in a covariance matrix of the form
$$ \Sigmabtildehat = \Sigmab + \Epsilonb + \Etab_N = \diag(\Sigmab_1,\ldots,\Sigmab_L)  + \Epsilonb + \Etab_N \;.$$
We can assume that the covariance matrix behaves in this convenient way because we can always obtain this structure using a permutation matrix. Therefore, (\ref{eq:blocks}) with $\Epsilonb \neq \bm{0}$ is equivalent to
$$\underbrace{\Sigmabtilde_1}_{\kappa_1\times\kappa_1}  \epsind   \ldots \epsind \underbrace{\Sigmabtilde_L}_{\kappa_L\times\kappa_L} \; .$$
Due to the block-structure, the eigenvectors of $\Sigmab$ are of shape
\begin{equation}
\label{eq:blockeigenstructure}\underbrace{\begin{pmatrix}  \bm{\ast}_{\kappa_{1}} \\ \bm{0} \end{pmatrix},\ldots,\begin{pmatrix}  \bm{\ast}_{\kappa_{1}} \\ \bm{0} \end{pmatrix}}_{\kappa_{1}\text{-many}}  ,    
\underbrace{\begin{pmatrix} \bm{0} \\ \bm{\ast}_{\kappa_{2}} \\ \bm{0} \end{pmatrix},\ldots,\begin{pmatrix} \bm{0} \\ \bm{\ast}_{\kappa_{2}}\\ \bm{0} \end{pmatrix}}_{\kappa_{2}\text{-many}}  , \ldots , 
\underbrace{\begin{pmatrix} \bm{0} \\ \bm{\ast}_{\kappa_{L}}  \end{pmatrix},\ldots,\begin{pmatrix} \bm{0} \\ \bm{\ast}_{\kappa_{L}} \end{pmatrix}}_{\kappa_{L}\text{-many}}  \;\;, 
\end{equation}
where $\bm{\ast}_{\kappa_\ell}$ with $\ell\in\{1,\ldots,L\}$ are vectors of length $\kappa_\ell$ and $\bm{0}$ are vectors of suitable dimension containing zeros. The first $\kappa_{b_1}$ eigenvectors have (at least) $M-\kappa_{b_1}$ zero-components, the following $\kappa_{b_2}$ eigenvectors have (at least) $M-\kappa_{b_2}$ zero-components and so on. The eigenvectors of $\Sigmabtildehat$ follow the same shape however being slightly perturbed due to $\Epsilonb$ and distorted by the noise $\Etab_N$.\\
\noindent
PLA for discarding, say, $A$ blocks $\Sigmab_{b_1},\ldots,\Sigmab_{b_A}$ with $\Sigmab_{b_a} \cind \Sigmab_{b_l}$ $\forall l \neq a$ for $a\in\{1,\ldots,A\}$ is provided in the following algorithm:
\begin{Algorithm}[\textbf{PLA}]\label{alg:PLAcorrBlock}
Discard the variables corresponding to $\Sigmab_{b_1},\ldots,\Sigmab_{b_A}$ according to PLA proceeds as follows:
\begin{enumerate}
\item[1.] Check if the eigenvectors of $\Sigmab$ satisfy the required structure in (\ref{eq:blockeigenstructure}) to discard $\Sigmab_{b_1} , \ldots , \Sigmab_{b_A}$.

\item[2.] Decide if $\Sigmab_{b_1} , \ldots , \Sigmab_{b_A}$ are relevant according to the explained variance $(\ref{eq:expvar})$ of the realisations $\xb_j$ of their contained random variables $X_j$ by calculating
\begin{equation}
\label{eq:algPLAcorrBlocksExpVar}
\dfrac{\sum_j\hat{\tilde{\lambda}}_j}{ \sum_{m} \hat{\tilde{\lambda}}_m } \;\;\; , 
\end{equation}
where $m\in\{1,\ldots,M\}$ and $j$ indexes all $X_j$ contained in $\Sigmab_{b_1},\ldots,\Sigmab_{b_A}$. 

\item[3.] Discard $\Sigmab_{b_1} , \ldots , \Sigmab_{b_A}$.
\end{enumerate}
\end{Algorithm}\noindent
In the spirit of PCA, during the first step of Algorithm~\ref{alg:PLAcorrBlock} we check if all relevant components of an eigenvector are below a given threshold $\tau$ in absolute terms. The purposes of this cut-off value is twofold: we want to detect the required structure despite the presence of noise and, due to the fuzziness of $\tau$, we detect the $\varepsilon$-uncorrelated variables. In Section~\ref{s:simulationstudy} we provide recommendations for $\tau$. The convergence rate for the sample eigenvectors and sample eigenvalues are provided in the following theorem. 
\begin{Theorem}
\label{th:BlockPLAcorrEps}
Let Assumption \ref{ass:1} and \ref{ass:2} hold. When dropping the variables corresponding to $\Sigmab_{b_a}$ according to PLA it holds that there are $\kappa_{b_a}$ eigenvectors of $\Sigmabhat$ of the form
$$ \vbtildehat_{j} =  \vb_{j} + \epsilonb_{j} + \etab_{j|N}  =  \begin{pmatrix}
\bm{0} &  v_{j}^{(M_{a-1})} & \cdots & v_{j}^{(M_a)}  & \bm{0}
\end{pmatrix}^\top + \epsilonb_{j}  +\etab_{j|N}$$
for $j\in\{1,\ldots,\kappa_{b_a}\}$ with corresponding eigenvalues
$$ \lambdatildehat_{j} = \lambdatilde_j + \eta_{j|N}^\lambda \;. $$
Then it holds that
$$\eta_{j|N}^\lambda = \bigO\left( \frac{1}{N^w} \right)  \;\text{ and }\; \etab_{j|N}  =\bigO\left( \frac{1}{\sqrt{N}} \right) \;,\; w \in[\tfrac{1}{2},\infty) \;.$$
\end{Theorem}\noindent
\begin{proof}
$ \eta_{j_a|N}^\lambda = \bigO\left( 1/N^w\right)$ with $w\in[\tfrac{1}{2},\infty)$ is due to Lemma~\ref{l:2} and $\etab_{j|N}=\bigO\left(1/\sqrt{N}\right)$ is a conclusion of \refNeudecker.
\end{proof}\noindent
Sufficient bounds for the perturbations that have to hold in order to satisfy correct discarding are given in the following theorem. It provides an intuition of the possible magnitude of the perturbations that results in a drop. Note that the bounds are sufficient and not necessary and sufficient.
\begin{Theorem}\label{th:Kahan}
Denote $\tilde{\lambda}_0 = \lambda_0  \equiv \infty$ and $\tilde{\lambda}_{M+1}  = \lambda_{M+1} \equiv -\infty$. For $j\in\{1,\ldots,M\}$ it holds that
$$ \dfrac{2^{3/2} \|\Epsilonb + \Etab_N \|_F }{\min(\lambda_{j-1} - \lambda_j,\lambda_j - \lambda_{j+1}) } < \tau\;\Rightarrow\;  \|\epsilonb_j + \etab_\jN \|_\infty < \tau \; . $$
\end{Theorem}
\begin{proof}
From Corollary 1 in \citep{YW15} we can conclude that $\|\epsilonb_j + \etab_\jN \|_2 \leq  2^{3/2} \|\Epsilonb + \Etab_N \|_F / \min(\lambda_{j-1} - \lambda_j,\lambda_j - \lambda_{j+1})$ which yields our desired result since $\|\epsilonb_j + \etab_\jN \|_\infty \leq \|\epsilonb_j + \etab_\jN\|_2$. 
\end{proof}\noindent
Hence, discarding is ensured if the perturbations are little conditioning on that the eigenvalues are not close. A next step is naturally to investigate what assures discarding in case of large perturbations and what causes close eigenvalues. We provide a brief intuition to answer those questions, however it is subject to future research. We cover the population case with $\Etab_N = \bm{0}$ and consider that $\lambda_{\ell-1}$, $\lambda_\ell$ and $\lambda_{\ell+1}$ are eigenvalues of $\Sigmab_{\overline{\ell}}$, $\Sigmab_{\ell-}$ and $\Sigmab_{\underline{\ell}}$ respectively. If needed, we ensure the validity of subtraction of quadratic matrices with different dimensions by padding the matrix with smaller dimension of concern, say, $\Ab$ in a way that $\Ab \leadsto \begin{pmatrix}
\Ab & \bm{0} \\ \bm{0} & \Deltab
\end{pmatrix}$ with $\Deltab = \diag(\delta,\ldots,\delta)$ and $\delta\to 0$ with $\delta>0$ is small.
\begin{Remark}
$\|\epsilonb_j \|_\infty < \tau$ is satisfied if 
\begin{equation}\label{eq:Remark}
\|\Epsilonb  \|_F < \tau \cdot D_{\overline{l},\underline{l}} 
\end{equation}
where $D_{\overline{\ell},\underline{\ell}}$ is the largest element lying in the Gershgorin discs of $\Sigmab_{\overline{\ell}} -\Sigmab_{\ell-} $ or $\Sigmab_{\ell-} - \Sigmab_{\underline{\ell}}$.
\end{Remark}\noindent
Hence, the more the correlation structure of $\Sigmab_{\ell-}$ differs from the other blocks, the more likely we are to ensure correct dropping.
\section{Decision about Discarding: The Choice of Thresholds}
\label{s:Decision}
\noindent
In this section we take a deeper look at the second step of Algorithm~\ref{alg:PLAcorrBlock}. We recap the idea behind $(\ref{eq:algPLAcorrBlocksExpVar})$ and the choice of the cut-off value $\tau$ respectively. Further, we provide an extension of $(\ref{eq:algPLAcorrBlocksExpVar})$ used in Algorithm~\ref{alg:PLAcorrBlock}.\\
Some steps of PLA correspond with PCA by construction, the choice of optimal thresholds is such an intersection. Therefore, we can utilize established approaches from PCA for this particular step. An elaborate overview of common techniques to evaluate the importance of each principal component (PC) in PCA is provided in \citep{JO02,PNJS05}. In this work, we evaluate the explained variance provided by the blocks we consider to discard and proceed to drop if the explained variance is small. This corresponds to the PCA method when checking the cumulative percentage of total variation explained by the PCs one wants to keep. While the contribution of the PCs is desired to be large in PCA, however, we want the contribution of the blocks in PLA to be small. The threshold in PCA depends on practical details of a particular data set as well as on the intention of the applicant and a rule of thumb is to choose the share of the explained variance by the PCs to be between $70\%-90\%$ \citep{JO02}. We suggest this procedure for PLA to decide in practice if the variables of consideration explain little of the overall variance, that is to say if the set of the reduced variables explains an acceptable amount of the overall variance. \\
The threshold $\tau$ regarding the cut-off of elements vary among applications. \citep{PNJS03} recaps some cut-off values from published studies which lie between $0.3$ and $0.5$. However, in Section~\ref{s:simulationstudy} we discovered that PLA performs well even for smaller thresholds than the ones mentioned above.\\
It is worth mentioning that we can obtain a measure of the contribution to the explained variance by combining both methods. Therefore, we consider again the case to discard $A$ blocks $\Sigmab_{b_1},\ldots,\Sigmab_{b_A}$ with corresponding eigenvectors $\vbtildehat_j$ and eigenvalues $\lambdatildehat_j$ respectively. For convenience purposes, we assume that the contained random variables are $X_j$ with $j\in\{1,\ldots,\kappa_A\}$. The explained variance of $\sum_j \lambdatildehat_j$ is then given by (\ref{eq:algPLAcorrBlocksExpVar}). However, note that $\hat{\tilde{v}}_j^{(k)}$, $k\notin\{1,\ldots,\kappa_A\}$, does not necessarily equal zero due to the perturbation. The eigenvectors are dominated by the random variables we consider to discard, but the remaining random variables also distort to some degree. Further, the elements of the complementary eigenvectors $\hat{\tilde{v}}_k^{(j)}$ are not necessarily equal to zero either. Hence, the $X_j$ might also slightly influence the covariance matrix in further directions. Since the eigenvectors are normalized, we can consider each squared eigenvector element as the portion of the corresponding random variable towards this eigenvector direction. Hence, the measure of the contribution to the explained variance is in spirit of (\ref{eq:algPLAcorrBlocksExpVar}) given by
$$ \left(\sum\limits_m \lambdatildehat_m\right)^{-1}   \left( \sum\limits_{j\in\{1,\ldots,\kappa_A\}}  \lambdatildehat_j \sum\limits_{i=1}^{\kappa_A} (\hat{\tilde{v}}_j^{(i)})^2    +  \sum\limits_{k\notin\{1,\ldots,\kappa_A\}}  \lambdatildehat_k \sum\limits_{i=1}^{\kappa_A}  (\hat{\tilde{v}}_k^{(i)})^2     \right) \; .$$ 
In practice, we can then adjust $\tau$ in order to increase or decrease this explained variance sufficient for application.
\section{Simulation Study}
\label{s:simulationstudy}
\noindent
Finding an optimal threshold $\tau$ theoretically is rather difficult due to the fuzziness of $\tau$. However, we conducted a simulation study where the population $\Xb$ consisting of $M$ variables with $\num{100000}$ realisations was simulated $S=\num{10000}$ times. By construction, $k\in\{1,\ldots,5\}$ variables were uncorrelated or a $\kappa\times\kappa$ block was uncorrelated with $\kappa\in\{2,\ldots,6\}$. Then, for each $S$ a sample $\xb$ of size $N = \num{10000}$ had been drawn and we conducted PLA for $\tau\in\{0.2,0.3,0.4,0.5,$ $0.6,0.7,0.8\}$. For $\tau$, we considered cut-off values used in published studies as an orientation \citep{PNJS03}. However, we also focused on tighter values because PLA appeared to perform well for small thresholds.\footnote{During research we also considered more extreme cut-off values like 0.1 or 0.9. However, PLA performs worse at the tails of $\tau$ which is also indicated by the following results. Hence, we chose $0.2$ and $0.8$ as sufficient borders for this publication.} \\
Since the covariance provided uncorrelated patterns by construction, we were able to calculate the type I errors as the share of iterations where PLA did not lead to a consideration of a drop i.e., when step 1 in Algorithm~\ref{alg:PLAcorrBlock} was not fulfilled. The results can be found in Appendix~\ref{a:Tables} in Table~\ref{tab:alphaN10000} for the variable case and in Table~\ref{tab:alphablockN10000} for the block case respectively. In a type I error sense, PLA performs best for the cut-off values in the center. However, tighter thresholds diminish the likelihood of conducting a type II error. Hence, we recommend to use a threshold of 0.3 when the blocks are of dimension one and a threshold of 0.4 otherwise because they still yield satisfying results while being a tight choice. Note that the explained variance has not been considered in this simulation since we were only interested in the correct detection of uncorrelation. \\
As indicated in (\ref{eq:Remark}), the threshold for $\varepsilon$-uncorrelated blocks depends on many factors and hence the type I errors are hard to simulate. As mentioned previously, we provide a starting point for further research regarding the search for optimal cut-off values in Section~\ref{s:Conclusion}. Still, the given values for $\tau$ do work as we demonstrate in Section~\ref{s:DataAnalysis}.
\section{Example}
\label{s:DataAnalysis}
\noindent
We provide an example of PLA for the case of an uncorrelated block and for the case of a single $\varepsilon$-uncorrelated variable in this section. The examples are based on simulated data sets and all values are rounded to two decimal places. \\
\noindent
In order to conduct PLA on an uncorrelated block, we simulated the population $\Xb$ consisting of $\num{100000}$ realisations of ten variables $X_1,\ldots,X_{10}$ such that $\corr(X_j,X_i) = 0$ for $j\in\{1,2\}$,$\;i\in\{3,\ldots,10\}$ and $\corr(X_1,X_2)\neq0$. The sample $\xb$ with corresponding sample correlation matrix $\Sigmabhat$ was constructed by drawing $\num{10000}$ observations without replacement from $\Xb$. $\Sigmabhat$ can be found in Appendix~\ref{a:Example} with
$$\Vbhat = \begin{pmatrix}
0.00   &   0.06   &  -0.69   &  -0.04   &  -0.01   &   0.01   &   0.00   &  -0.01   &   0.72   &   0.01 \\
0.00   &   0.06   &  -0.72   &  -0.04   &  -0.01   &   0.00   &   0.00   &   0.01   &  -0.69   &   0.00 \\
-0.18   &  -0.11   &  -0.01   &  -0.03   &   0.18   &   0.08   &  -0.50   &   0.79   &   0.02   &  -0.20 \\
-0.37    &  0.05   &   0.05   &  -0.75   &  -0.35   &  -0.31   &   0.00   &  -0.07   &   0.00   &  -0.26 \\
-0.32   &  -0.35   &  -0.02   &   0.03   &  -0.31   &   0.78   &  -0.05   &  -0.20   &   0.00   &  -0.17 \\
-0.45   &   0.53   &   0.04   &   0.20   &  -0.35   &   0.09   &   0.23   &   0.32   &   0.00   &   0.44 \\
-0.23    &  0.17   &   0.00   &   0.26   &  -0.10   &  -0.21   &  -0.77   &  -0.43   &  -0.01   &   0.12 \\
-0.58   &  -0.54   &  -0.07   &   0.31   &   0.18   &  -0.41   &   0.27   &  -0.02   &   0.00   &   0.03 \\
-0.23    &  0.49   &   0.02   &   0.26   &   0.28   &   0.05   &   0.17   &  -0.12   &   0.00   &  -0.72 \\
-0.30   &   0.13   &   0.02   &  -0.40   &   0.71   &   0.27   &  -0.02   &  -0.18   &  -0.01   &   0.37 \\
\end{pmatrix} \; $$
and
$$\diag(\Lambdabhat) = \begin{pmatrix} 154.97  &  26.70  &  23.93  &  19.74  &   9.05  &   6.24  &   2.48  &   1.68  &   0.99  &   0.22  \end{pmatrix} .$$
$\Vbhat$ reveals that $X_1$ and $X_2$ are dominant in $\vbhat_3$ and $\vbhat_9$ however influence the remaining eigenvectors only marginally hence primarily distort $\Sigmab$ in only two directions. We then consider to drop the block containing $X_1$ and $X_2$ denoted $\Sigmab_1$ with $\kappa_1 = 2$ represented by $\vbhat_3$ and $\vbhat_9$. We choose $\tau = 0.4$ following our argument in Section~\ref{s:simulationstudy}.

\begin{enumerate}
\item[1.] \textit{Check if the eigenvectors of $\Sigmab$ satisfy the required structure in (\ref{eq:blockeigenstructure}) to discard $\Sigmab_{1}$.} \\
For $\hat{\vb}_3$ and $\hat{\vb}_9$, it holds that all components in absolute terms are below the given threshold except the first two components.

\item[2.] \textit{Decide if $\Sigmab_{1}$ is relevant according to the explained variance $(\ref{eq:expvar})$ of the realisations $\xb_1$ and $\xb_2$ of their contained random variables $X_1$ and $X_2$.}\\
It holds that $\Sigmabhat_1$ explains $(\lambdahat_3 + \lambdahat_{9})/\sum_m\lambdahat_m \approx 0.1013$ hence $10.13\%$ of the overall variance which is a fairly small amount.

\item[3.] \textit{Discard $\Sigmab_{1}$}.\\
We drop $X_3$ and $X_9$.
\end{enumerate}
\noindent
To provide an example for a single $\varepsilon$-uncorrelated variable, we simulated the population $\Xb$ consisting of $\num{100000}$ realisations of ten variables $X_1,\ldots,X_{10}$ such that $\corr(X_1,X_i) = \textit{small}$ for $i\in\{2,\ldots,10\}$. The sample $\xb$ was constructed by drawing $\num{10000}$ observations without replacement from $\Xb$. The first row of $\Sigmabtilde$ displaying $\Cov(X_1,X_i)$ is
$$\begin{pmatrix} 11.01 & -1.22 &  0.26 & -0.91 &  0.51 &  0.04 &  0.53 & -1.33 & -0.29 & -0.93  \end{pmatrix}$$
and the corresponding sample counterpart is given by
$$\begin{pmatrix} 10.94 & -1.22 &  0.50 & -0.83 &  0.52 &  0.30 &  0.48 & -1.05 & -0.28 & -0.89 \end{pmatrix} .$$
The whole matrices can be found in Appendix~\ref{a:Example}. This example emphasizes the difficulty, whether small values in $\Sigmabtildehat$ are due to uncorrelation or due to $\varepsilon$-uncorrelation. The eigendecomposition of $\Sigmabtildehat$ yields
$$\Vbtildehat = \begin{pmatrix}
 0.00 & 0.00 &  0.10 &  0.10 &  0.94 &  0.29 & -0.12 & -0.04 &  0.01 &  0.03 \\
-0.30 &  0.33 & -0.31 & -0.13 & -0.06 &  0.46 &  0.42 & -0.38 &  0.04 &  0.39 \\
-0.35 & -0.57 &  0.22 &  0.50 & -0.02 & -0.19 &  0.17 & -0.23 &  0.03 &  0.36 \\
-0.38 & -0.01 & -0.69 &  0.24 &  0.13 & -0.19 & -0.01 &  0.37 & -0.35 & -0.11 \\
-0.29 &  0.24 &  0.52 &  0.03 & -0.12 &  0.26  & 0.00  & 0.57  & -0.34  & 0.26 \\
-0.31 & -0.20 &  0.23 & -0.34 &  0.10 & -0.02 &  0.53 & -0.09 & -0.21 & -0.60 \\
-0.30  & 0.58 &  0.16 &  0.35 &  0.07  & -0.34  & 0.11  &-0.01  & 0.49  & -0.22 \\
-0.43 & -0.16 & -0.04 & -0.63 &  0.11 & -0.26 & -0.25 &  0.16 &  0.38 &  0.27 \\
-0.18 &  0.28 &  0.15 & -0.11 &  0.03 & -0.34 & -0.38 & -0.53 & -0.55 &  0.06 \\
-0.39 & -0.14 & -0.02 &  0.15 & -0.24 &  0.52 & -0.52 & -0.15 &  0.16 & -0.39
\end{pmatrix} \; $$
and
$$\diag(\Lambdabtildehat) = \begin{pmatrix} 267.69 &  43.74 &  25.22  & 19.76 &  11.00 &   8.79 &   5.87 &   1.04 &   0.32 &   0.10  \end{pmatrix}.$$
From $\Vbtildehat$ it appears that $X_1$ distorts $\Sigmab$ primarily via $\vb_5$ and hardly influences the other directions. Hence, we consider to drop $X_1$ where we choose $\tau = 0.3$ according to Section~\ref{s:simulationstudy}.
\begin{enumerate}
\item[1.] \textit{Check if the eigenvectors of $\Sigmab$ satisfy the required structure in (\ref{eq:blockeigenstructure}) to discard $X_1$.} \\
It holds that all components except the first one of $\hat{\tilde{\vb}}_5$ are below the given threshold in absolute terms.

\item[2.] \textit{Decide if $X_1$ is relevant according to the explained variance $(\ref{eq:expvar})$ of the realisations $\xb_1$.}\\
It holds that $\xb_1$ explains $\lambdatildehat_5/\sum_m\lambdatildehat_m \approx 0.0287$ hence $2.87\%$ of the overall variance which is a fairly small amount.

\item[3.] \textit{Discard $X_1$}.\\
We drop $X_1$.
\end{enumerate}

\section{Concluding Remarks and Outlook}
\label{s:Conclusion}
\noindent
PLA is a tool for dimension reduction. We have shown the different covariance matrix structures needed to conduct PLA and provided bounds for the sample covariance, sample eigenvectors and sample eigenvalues. Based on our simultation study, we suggest using a cut-off threshold of 0.3 for the eigenvectors to detect if variables or blocks of variables can be discarded.\\
As an extension, we will compare the regression performance of PLA reduced data not only with Ordinary Least Square (OLS) regression but also with PCA regression as well as with PCA-methods reduced regression in an upcoming work. For instance, the natural link between OLS regression and PCA is given since PCA can be seen as a minimization problem in OLS sense:
$$ \min\limits_{\betab} \| \xitilde^{(i)} - \betab^\top\Xb  \|_2  =  \min\limits_{\betab} \|\xitilde^{(i)} - \betab^\top \Vbtilde\xibtilde  \|_2  $$
is minimized for $\betab =\vbtilde_i$ hence the loadings minimize the distance between the subspace spanned by the $K$ PCs and the original space. As an appetizer, we will address the link between OLS and PLA in this outlook however for the special case that $\Sigmabtilde = \Sigmab$ hence with $\Epsilonb = \bm{0}$. In a general regression problem 
$$ \yb = \xb\betab + \bm{u} $$
where $\yb$ denotes the $N\times1$ vector of dependent observations of a random variable $Y$ and $\bm{u}$ denotes a $N\times1$ vector of unobserved random errors, the well known $M\times1$ OLS estimator is given by
$$\hat{\betab}_{\text{OLS}} \equiv (\xb^\top\xb)^{-1}\xb^\top\bm{y} = \Sigmabhat^{-1}\begin{pmatrix} \hat{\Cov}(X_1,Y) \\ \vdots \\ \hat{\Cov}(X_M,Y) \end{pmatrix} \PC \Sigmab^{-1} \begin{pmatrix} \Cov(X_1,Y) \\ \vdots \\ \Cov(X_M,Y) \end{pmatrix}  \; . $$
If we consider a sample $\xb_{\yb} \equiv \begin{pmatrix} \yb & \xb_1 & \cdots & \xb_M \end{pmatrix}$ containing $\xb$ and $\yb$ of a random vector $\Xb_Y \equiv \begin{pmatrix} Y & X_1 & \cdots & X_M \end{pmatrix}$ and assume that, say, $X_1$ is discarded by PLA, it holds that 
$$ \betab_{\text{OLS}} =
\begin{pmatrix}
\tfrac{1}{\sigma_{11}} & \bm{0} &  \\ \bm{0} & \bm{\ast}  
\end{pmatrix}
\begin{pmatrix} 0 \\ \bm{\ast} \end{pmatrix}  =\begin{pmatrix}  0 \\ \bm{\ast} \end{pmatrix} $$
hence the first component of $\betab_{\text{OLS}}$ is zero. With $\Epsilonb \neq \bm{0}$ however, the first component is different from zero which is caused by $\Epsilonb$. To find an optimal threshold subject to $\Epsilonb$, we investigate the regression performance of $\betab_{\text{OLS}}$ depending on $\tau$ and $\Epsilonb$. \\
Further, we consider different mechanics for the choice of $\tau$. This may imply a combination of an upper and lower threshold or a flexible cut-off value depending on the ratio of discarded variables with respect to the overall amount of variables $M$, etc. We also take into account that uncorrelated variables cause zero components in the rows as well as in the columns of $\Vb$ and investigate if checking both, rows and columns, might increase PLA performance. Checking for type II errors corresponding to $\tau$ is also part of this work.\\
One step beyond: instead of choosing a certain cut-off value one could implement a penalty term, such as an elastic net or lasso as a special case, to diminish the absolute value of the loading elements. By construction, this shrinks the loadings towards zero \citep{JC16,ZHT06}.\\
We will also address the case if the dimensionality of $\xb$ is such that $N/M\to\theta\in(0,\infty)$ hence that $N$ is not necessarily larger than $M$ and that $M$ is possibly large in general. We will investigate if PLA is still a feasible tool for dimension reduction.

\appendix

\section{Complementary Results for Section~\ref{s:DataAnalysis}}
\label{a:Example}
\noindent
In the following, we provide complement material for the examples in Section~\ref{s:DataAnalysis}. All values are rounded to two decimal places. We provide $\Sigmabtilde$ for the example of dropping a single block

\subsection{Complementary Results for Section~\ref{s:DataAnalysis}}
\noindent
$$\Sigmabhat=
\begin{pmatrix}
12.06 & 11.46 &  0.01 & -0.08 & -0.12 &  0.11 &  0.08 &  0.12 &  0.20 &  0.10 \\
11.46 & 12.86 & -0.03 & -0.21 & -0.23 & -0.01 &  0.01 & -0.05 &  0.13 &  0.01 \\
0.01 & -0.03 &  7.51 &  9.82 &  9.83 & 10.41 &  6.04 & 17.51 &  4.90 &  9.29 \\
-0.08 & -0.21 &  9.82 & 33.77 & 16.85 & 23.99 & 10.46 & 27.48 &  8.86 & 20.14 \\
-0.1 &2 -0.23 &  9.83 & 16.85 & 24.25 & 18.89 &  9.84 & 31.79 &  6.52 & 12.73 \\
0.11 & -0.01 & 10.41 & 23.99 & 18.89 & 40.59 & 19.17 & 32.93 & 22.72 & 18.64 \\
0.08 &  0.01 &  6.04 & 10.46 &  9.84 & 19.17 & 12.77 & 20.05 & 11.23 &  8.55 \\
0.12 & -0.05 & 17.51 & 27.48 & 31.79 & 32.93 & 20.05 & 63.25 & 15.29 & 22.63 \\
0.20 &  0.13 &  4.90 &  8.86 &  6.52 & 22.72 & 11.23 & 15.29 & 16.67 & 12.04 \\
0.10 &  0.01 &  9.29 & 20.14 & 12.73 & 18.64 &  8.55 & 22.63 & 12.04 & 22.26
\end{pmatrix} 
$$

as well as $\Sigmabtilde$ and $\Sigmabtildehat$ for the example when discarding a single $\varepsilon$-uncorrelated variable respectively.

$$\Sigmabtilde = \begin{pmatrix}
11.01 & -1.22 &  0.26 & -0.91 &  0.51 &  0.04 &  0.53 & -1.33 & -0.29 & -0.93 \\
-1.22 & 34.19 & 16.81 & 33.71 & 23.65 & 22.54 & 29.05 & 32.69 & 15.30 & 30.21 \\
0.26 & 16.81 & 54.58 & 34.41 & 24.94 & 33.69 & 19.44 & 39.46 &  9.99 & 41.10 \\
-0.91 & 33.71 & 34.41 & 51.63 & 20.69 & 26.62 & 29.92 & 42.55 & 15.44 & 39.61 \\
0.51 & 23.65 & 24.94 & 20.69 & 33.63 & 25.59 & 31.17 & 31.54 & 17.77 & 31.17 \\
0.04 & 22.54 & 33.69 & 26.62 & 25.59 & 33.90 & 19.92 & 41.76 & 13.19 & 31.69 \\
0.53 & 29.05 & 19.44 & 29.92 & 31.17 & 19.92 & 42.91 & 27.96 & 22.01 & 27.61 \\
-1.33 & 32.69 & 39.46 & 42.55 & 31.54 & 41.76 & 27.96 & 61.35 & 21.28 & 44.78 \\
-0.29 & 15.30 &  9.99 & 15.44 & 17.77 & 13.19 & 22.01 & 21.28 & 14.93 & 16.50 \\
-0.93 & 30.21 & 41.10 & 39.61 & 31.17 & 31.69 & 27.61 & 44.78 & 16.50 & 47.81
\end{pmatrix} \; , $$

$$ \Sigmabtildehat = \begin{pmatrix}
10.94 & -1.22 &  0.50 & -0.83 &  0.52 &  0.30 &  0.48 & -1.05 & -0.28 & -0.89 \\
-1.22 & 34.20 & 16.60 & 33.46 & 23.53 & 22.17 & 29.01 & 31.95 & 15.04 & 30.01 \\
0.50 & 16.60 & 54.60 & 34.43 & 24.43 & 33.21 & 19.13 & 38.56 &  9.57 & 40.77 \\
-0.83 & 33.46 & 34.43 & 51.43 & 20.24 & 26.09 & 29.62 & 41.52 & 14.96 & 39.28 \\
0.52 & 23.53 & 24.43 & 20.24 & 33.52 & 24.97 & 31.31 & 30.51 & 17.62 & 30.73 \\
0.30  & 22.17 & 33.21 & 26.09 & 24.97 & 33.29 & 19.29 & 40.74 & 12.69 & 31.08 \\
0.48 & 29.01 & 19.13 & 29.62 & 31.31 & 19.29 & 43.60 & 26.90 & 22.07 & 27.27 \\
-1.05 & 31.95 & 38.56 & 41.52 & 30.51 & 40.74 & 26.90 & 59.78 & 20.54 & 43.69 \\
-0.28 & 15.04 &  9.57 & 14.96 & 17.62 & 12.69 & 22.07 & 20.54 & 14.84 & 16.03 \\
-0.89 & 30.01 & 40.77 & 39.28 & 30.73 & 31.08 & 27.27 & 43.69 & 16.03 & 47.33
\end{pmatrix} \; .
$$

\section{Threshold Values}
\label{a:Tables}\noindent
We provide the type I error rates for single uncorrelated variables and for an uncorrelated block respectively. As specified in Section~\ref{s:simulationstudy}, the probabilities are calculated as the share of iterations where PLA did not lead to a consideration of a drop despite it should have been considered by construction.

\begin{ThreePartTable}
\begin{TableNotes}
\footnotesize
\item \textit{Notes:} We computed the type I errors as the share of iterations where the variable has not been discarded.
\end{TableNotes}
\begin{longtable}{rllllllll}
\caption{\label{tab:alphaN10000}Type I error for $k$ uncorrelated variables with sample size $N=\num{10000}$.} \\
\toprule
\multicolumn{9}{c}{$N=\num{10000}$}  \\
\cmidrule(l{3pt}r{3pt}){1-2} \cmidrule(l{3pt}r{3pt}){3-9}
 $M$  & $k$ & $\tau=0.2$ & $\tau=0.3$ & $\tau=0.4$ & $\tau=0.5$ & $\tau=0.6$ & $\tau=0.7$ & $\tau=0.8$\\
\midrule
\endfirsthead
\caption[]{Type I error for $k$ uncorrelated variables with sample size $N=\num{10000}$ \textit{(continued)}.}\\
\toprule
\multicolumn{9}{c}{$N=\num{10000}$}  \\
\cmidrule(l{3pt}r{3pt}){1-2} \cmidrule(l{3pt}r{3pt}){3-9}
 $M$  & $k$ & $\tau=0.2$ & $\tau=0.3$ & $\tau=0.4$ & $\tau=0.5$ & $\tau=0.6$ & $\tau=0.7$ & $\tau=0.8$\\
\midrule
\endhead
\
\endfoot
\bottomrule
\insertTableNotes
\endlastfoot
   10 & 1 &  0.0656 &  0.0306 &  0.0108 &  0.0011 &  0.0000 &  0.0001 &  0.0155\\
   20 & 1 &  0.0919 &  0.0292 &  0.0033 &  0.0000 &  0.0000 &  0.0007 &  0.0274\\
   30 & 1 &  0.0941 &  0.0142 &  0.0004 &  0.0000 &  0.0001 &  0.0020 &  0.0426\\
   40 & 1 &  0.0867 &  0.0088 &  0.0003 &  0.0000 &  0.0000 &  0.0027 &  0.0468\\
   50 & 1 &  0.0782 &  0.0038 &  0.0000 &  0.0000 &  0.0001 &  0.0037 &  0.0623\\
   60 & 1 &  0.0700 &  0.0020 &  0.0000 &  0.0000 &  0.0001 &  0.0053 &  0.0647\\
   70 & 1 &  0.0573 &  0.0007 &  0.0000 &  0.0000 &  0.0001 &  0.0072 &  0.0738\\
   80 & 1 &  0.0451 &  0.0003 &  0.0000 &  0.0000 &  0.0002 &  0.0102 &  0.0849\\
   90 & 1 &  0.0349 &  0.0002 &  0.0000 &  0.0000 &  0.0002 &  0.0136 &  0.0917\\
 100 & 1 &  0.0254 &  0.0001 &  0.0000 &  0.0000 &  0.0006 &  0.0136 &  0.0971 \\
  \addlinespace
   10 & 2 &  0.1290 &  0.0656 &  0.0251 &  0.0099 &  0.0033 &  0.0005 &  0.0279\\
   20 & 2 &  0.1911 &  0.0648 &  0.0153 &  0.0051 &  0.0015 &  0.0019 &  0.0552\\
   30 & 2 &  0.1939 &  0.0436 &  0.0093 &  0.0048 &  0.0020 &  0.0031 &  0.0849\\
   40 & 2 &  0.1912 &  0.0313 &  0.0072 &  0.0043 &  0.0028 &  0.0067 &  0.0987\\
   50 & 2 &  0.1659 &  0.0215 &  0.0066 &  0.0037 &  0.0022 &  0.0085 &  0.1122\\
   60 & 2 &  0.1523 &  0.0143 &  0.0061 &  0.0044 &  0.0028 &  0.0099 &  0.1345\\
   70 & 2 &  0.1241 &  0.0130 &  0.0065 &  0.0048 &  0.0019 &  0.0176 &  0.1441\\
   80 & 2 &  0.1091 &  0.0129 &  0.0087 &  0.0044 &  0.0024 &  0.0223 &  0.1557\\
   90 & 2 &  0.0904 &  0.0121 &  0.0073 &  0.0044 &  0.0036 &  0.0274 &  0.1791\\
 100 & 2 &  0.0768 &  0.0128 &  0.0067 &  0.0040 &  0.0025 &  0.0288 &  0.1893\\
  \addlinespace
   10 & 3 &  0.1997 &  0.1073 &  0.0533 &  0.0184 &  0.0067 &  0.0009 &  0.0411\\
   20 & 3 &  0.2887 &  0.1140 &  0.0332 &  0.0151 &  0.0052 &  0.0023 &  0.0886\\
   30 & 3 &  0.2916 &  0.0869 &  0.0228 &  0.0122 &  0.0069 &  0.0075 &  0.1232\\
   40 & 3 &  0.2845 &  0.0613 &  0.0210 &  0.0118 &  0.0078 &  0.0112 &  0.1457\\
   50 & 3 &  0.2581 &  0.0429 &  0.0245 &  0.0142 &  0.0070 &  0.0146 &  0.1785\\
   60 & 3 &  0.2362 &  0.0414 &  0.0212 &  0.0137 &  0.0062 &  0.0199 &  0.1994\\
   70 & 3 &  0.2086 &  0.0375 &  0.0197 &  0.0127 &  0.0056 &  0.0320 &  0.2212\\
   80 & 3 &  0.1848 &  0.0343 &  0.0192 &  0.0146 &  0.0077 &  0.0329 &  0.2426\\
   90 & 3 &  0.1641 &  0.0350 &  0.0211 &  0.0128 &  0.0059 &  0.0412 &  0.2564\\
 100 & 3 &  0.1366 &  0.0341 &  0.0210 &  0.0124 &  0.0086 &  0.0465 &  0.2741\\
  \addlinespace
   10 & 4 &  0.2485 &  0.1528 &  0.0837 &  0.0355 &  0.0120 &  0.0021 &  0.0475\\
   20 & 4 &  0.3802 &  0.1782 &  0.0608 &  0.0265 &  0.0117 &  0.0068 &  0.1118\\
   30 & 4 &  0.3877 &  0.1337 &  0.0498 &  0.0266 &  0.0121 &  0.0112 &  0.1577\\
   40 & 4 &  0.3901 &  0.0994 &  0.0447 &  0.0261 &  0.0119 &  0.0172 &  0.1936\\
   50 & 4 &  0.3625 &  0.0854 &  0.0403 &  0.0231 &  0.0113 &  0.0239 &  0.2269\\
   60 & 4 &  0.3255 &  0.0737 &  0.0436 &  0.0227 &  0.0118 &  0.0331 &  0.2627\\
   70 & 4 &  0.3041 &  0.0655 &  0.0445 &  0.0297 &  0.0113 &  0.0398 &  0.2870\\
   80 & 4 &  0.2668 &  0.0684 &  0.0403 &  0.0252 &  0.0129 &  0.0483 &  0.3003\\
   90 & 4 &  0.2356 &  0.0674 &  0.0418 &  0.0242 &  0.0147 &  0.0600 &  0.3306\\
 100 & 4 &  0.2187 &  0.0678 &  0.0414 &  0.0261 &  0.0133 &  0.0640 &  0.3502\\
 \addlinespace
   10 & 5 &  0.2897 &  0.1861 &  0.1028 &  0.0543 &  0.0229 &  0.0043 &  0.0594\\
   20 & 5 &  0.4695 &  0.2363 &  0.0966 &  0.0422 &  0.0219 &  0.0093 &  0.1380\\
   30 & 5 &  0.4909 &  0.1910 &  0.0722 &  0.0437 &  0.0182 &  0.0154 &  0.2057\\
   40 & 5 &  0.4807 &  0.1525 &  0.0711 &  0.0457 &  0.0205 &  0.0243 &  0.2431\\
   50 & 5 &  0.4661 &  0.1373 &  0.0724 &  0.0421 &  0.0224 &  0.0316 &  0.2754\\
   60 & 5 &  0.4218 &  0.1197 &  0.0746 &  0.0416 &  0.0231 &  0.0448 &  0.3164\\
   70 & 5 &  0.4037 &  0.1128 &  0.0692 &  0.0425 &  0.0225 &  0.0502 &  0.3439\\
   80 & 5 &  0.3552 &  0.1097 &  0.0644 &  0.0390 &  0.0207 &  0.0637 &  0.3729\\
   90 & 5 &  0.3214 &  0.1081 &  0.0676 &  0.0371 &  0.0193 &  0.0754 &  0.4002\\
 100 & 5 &  0.2912 &  0.1059 &  0.0706 &  0.0391 &  0.0218 &  0.0867 &  0.4150\\*
\end{longtable}
\end{ThreePartTable}

\begin{ThreePartTable}
\begin{TableNotes}
\footnotesize
\item \textit{Notes:} We computed the type I errors as the share of iterations where the block has not been discarded.
\end{TableNotes}
\begin{longtable}{rllllllll}
\caption{\label{tab:alphablockN10000}Type I error for an uncorrelated block containing $\kappa$ random variables with sample size $N=\num{10000}$}\\
\toprule
\multicolumn{9}{c}{$N=\num{10000}$}  \\
\cmidrule(l{3pt}r{3pt}){1-2} \cmidrule(l{3pt}r{3pt}){3-9}
 $M$  & $\kappa$ & $\tau=0.2$ & $\tau=0.3$ & $\tau=0.4$ & $\tau=0.5$ & $\tau=0.6$ & $\tau=0.7$ & $\tau=0.8$\\
\midrule
\endfirsthead
\caption[]{Type I error for an uncorrelated block containing $\kappa$ random variables with sample size $N=\num{10000}$ \textit{(continued)}}\\
\toprule
\multicolumn{9}{c}{$N=\num{10000}$}  \\
\cmidrule(l{3pt}r{3pt}){1-2} \cmidrule(l{3pt}r{3pt}){3-9}
 $M$  & $\kappa$ & $\tau=0.2$ & $\tau=0.3$ & $\tau=0.4$ & $\tau=0.5$ & $\tau=0.6$ & $\tau=0.7$ & $\tau=0.8$\\
\midrule
\endhead
\
\endfoot
\bottomrule
\insertTableNotes
\endlastfoot
       10  &   2  &  0.1279  &  0.0769  &  0.0470  &  0.0237  &  0.0107  &  0.0277  &  0.2722 \\
       20  &   2  &  0.2381  &  0.1483  &  0.0945  &  0.0437  &  0.0240  &  0.0582  &  0.3219 \\
       30  &   2  &  0.3075  &  0.1905  &  0.1201  &  0.0613  &  0.0295  &  0.0827  &  0.3459 \\
       40  &   2  &  0.3609  &  0.2312  &  0.1373  &  0.0692  &  0.0347  &  0.0975  &  0.3718 \\
       50  &   2  &  0.4030  &  0.2500  &  0.1520  &  0.0722  &  0.0416  &  0.1117  &  0.3841 \\
       60  &   2  &  0.4439  &  0.2821  &  0.1781  &  0.0865  &  0.0475  &  0.1169  &  0.3972 \\
       70  &   2  &  0.4611  &  0.3051  &  0.1898  &  0.0902  &  0.0499  &  0.1320  &  0.4079 \\
       80  &   2  &  0.4884  &  0.3198  &  0.1929  &  0.0991  &  0.0541  &  0.1434  &  0.4430 \\
       90  &   2  &  0.5244  &  0.3402  &  0.2022  &  0.0991  &  0.0619  &  0.1579  &  0.4427 \\
      100 &   2  &  0.5346  &  0.3550  &  0.2281  &  0.1048  &  0.0695  &  0.1658  &  0.4578 \\
\addlinespace
       10  &   3  &  0.1279  &  0.0770  &  0.0428  &  0.0189  &  0.0308  &  0.4265  &  0.7814 \\
       20  &   3  &  0.3234  &  0.1979  &  0.1116  &  0.0480  &  0.0845  &  0.4966  &  0.8171 \\
       30  &   3  &  0.4275  &  0.2678  &  0.1429  &  0.0675  &  0.1128  &  0.5333  &  0.8332 \\
       40  &   3  &  0.4978  &  0.3125  &  0.1718  &  0.0759  &  0.1400  &  0.5630  &  0.8475 \\
       50  &   3  &  0.5485  &  0.3539  &  0.1982  &  0.0863  &  0.1641  &  0.5880  &  0.8600 \\
       60  &   3  &  0.5846  &  0.3775  &  0.2209  &  0.1035  &  0.1824  &  0.6031  &  0.8665 \\
       70  &   3  &  0.6218  &  0.4129  &  0.2291  &  0.1052  &  0.1969  &  0.6185  &  0.8800 \\
       80  &   3  &  0.6498  &  0.4327  &  0.2394  &  0.1147  &  0.2096  &  0.6350  &  0.8843 \\
       90  &   3  &  0.6728  &  0.4450  &  0.2585  &  0.1211  &  0.2348  &  0.6426  &  0.8939 \\
      100 &   3  &  0.6960  &  0.4703  &  0.2615  &  0.1304  &  0.2430  &  0.6608  &  0.8982 \\\addlinespace
       10  &   4  &  0.1277  &  0.0781  &  0.0421  &  0.0173  &  0.2243  &  0.8128  &  0.9734 \\
       20  &   4  &  0.3676  &  0.2079  &  0.1057  &  0.0584  &  0.3072  &  0.8463  &  0.9780 \\
       30  &   4  &  0.5009  &  0.2993  &  0.1485  &  0.0908  &  0.3955  &  0.8717  &  0.9851 \\
       40  &   4  &  0.5915  &  0.3557  &  0.1773  &  0.1158  &  0.4410  &  0.8858  &  0.9876 \\
       50  &   4  &  0.6432  &  0.4125  &  0.2040  &  0.1344  &  0.4816  &  0.9049  &  0.9891 \\
       60  &   4  &  0.6851  &  0.4351  &  0.2222  &  0.1577  &  0.5089  &  0.9134  &  0.9921 \\
       70  &   4  &  0.7232  &  0.4778  &  0.2415  &  0.1793  &  0.5479  &  0.9167  &  0.9908 \\
       80  &   4  &  0.7512  &  0.5034  &  0.2488  &  0.1927  &  0.5766  &  0.9283  &  0.9926 \\
       90  &   4  &  0.7746  &  0.5265  &  0.2692  &  0.2145  &  0.6008  &  0.9325  &  0.9941 \\
      100 &   4  &  0.7961  &  0.5375  &  0.2829  &  0.2298  &  0.6239  &  0.9406  &  0.9945 \\
\addlinespace
       10  &   5  &  0.1691  &  0.0946  &  0.0511  &  0.0480  &  0.5810  &  0.9717  &  0.9988 \\
       20  &   5  &  0.3643  &  0.2074  &  0.1013  &  0.0889  &  0.6268  &  0.9740  &  0.9988 \\
       30  &   5  &  0.5535  &  0.3302  &  0.1521  &  0.1690  &  0.7147  &  0.9825  &  0.9991 \\
       40  &   5  &  0.6465  &  0.4036  &  0.1834  &  0.2369  &  0.7694  &  0.9862  &  0.9999 \\
       50  &   5  &  0.7113  &  0.4433  &  0.2144  &  0.2821  &  0.8043  &  0.9887  &  0.9996 \\
       60  &   5  &  0.7666  &  0.4891  &  0.2280  &  0.3216  &  0.8174  &  0.9910  &  0.9998 \\
       70  &   5  &  0.7916  &  0.5189  &  0.2475  &  0.3592  &  0.8494  &  0.9915  &  0.9997 \\
       80  &   5  &  0.8233  &  0.5478  &  0.2691  &  0.3933  &  0.8622  &  0.9926  &  1.0000 \\
       90  &   5  &  0.8424  &  0.5653  &  0.2842  &  0.4164  &  0.8703  &  0.9944  &  0.9996 \\
      100 &   5  &  0.8579  &  0.5896  &  0.2978  &  0.4455  &  0.8881  &  0.9952  &  1.0000 \\
\addlinespace
       10  &   6  &  0.1323  &  0.0757  &  0.0427  &  0.1465  &  0.8519  &  0.9974  &  1.0000 \\
       20  &   6  &  0.3691  &  0.2201  &  0.1003  &  0.1994  &  0.8730  &  0.9980  &  1.0000 \\
       30  &   6  &  0.5590  &  0.3270  &  0.1598  &  0.3095  &  0.9143  &  0.9980  &  1.0000 \\
       40  &   6  &  0.6804  &  0.4204  &  0.1988  &  0.4170  &  0.9370  &  0.9992  &  1.0000 \\
       50  &   6  &  0.7600  &  0.4844  &  0.2401  &  0.4925  &  0.9585  &  0.9998  &  1.0000 \\
       60  &   6  &  0.8019  &  0.5155  &  0.2635  &  0.5508  &  0.9626  &  0.9994  &  1.0000 \\
       70  &   6  &  0.8350  &  0.5481  &  0.2992  &  0.6064  &  0.9695  &  0.9999  &  1.0000 \\
       80  &   6  &  0.8604  &  0.5718  &  0.3295  &  0.6303  &  0.9759  &  0.9999  &  1.0000 \\
       90  &   6  &  0.8832  &  0.5943  &  0.3377  &  0.6744  &  0.9795  &  0.9996  &  1.0000 \\
      100 &   6  &  0.8992  &  0.6150  &  0.3533  &  0.6870  &  0.9820  &  1.0000  &  1.0000 \\*
\end{longtable}
\end{ThreePartTable}

\section*{References}
\bibliography{bib}{}
\bibliographystyle{abbrv} 

\end{document}